\documentclass[leqno,11pt]{amsart}
\usepackage[top=1.40in, bottom=1.40in, left=1.40in, right=1.40in]{geometry}
\usepackage{amscd,amsfonts,mathrsfs,amsthm,enumerate}
\usepackage{amssymb, amsmath,bigints}
\usepackage{stmaryrd}
\usepackage{epsfig}
\usepackage[sorted-cites,nobysame]{amsrefs}
\usepackage{latexsym}
\usepackage[usenames,dvipsnames]{color}
\usepackage[T1]{fontenc}
\usepackage{calligra}
\DeclareFontShape{T1}{calligra}{m}{n}{<->s*[2.3]callig15}{}
\DeclareMathAlphabet{\mathcalligra}{T1}{calligra}{m}{n}
\usepackage[colorlinks=true,linkcolor=blue,citecolor=BrickRed,urlcolor=RoyalBlue]{hyperref}
\newtheorem{theorem}{Theorem}[section]

\newtheorem{lemma}[theorem]{Lemma}
\newtheorem{corollary}[theorem]{Corollary}
\newtheorem{proposition}[theorem]{Proposition}
\newtheorem{definition}[theorem]{Definition}
\theoremstyle{definition}
\newtheorem*{assumption*}{$\lambda_{1}$-Condition}

\newtheorem{remark}[theorem]{Remark}

\DeclareMathOperator\Int{int} 

\def\B{\mathbb B}
\def\C{\mathbb C} 
\def\N{\mathbb N}
\def\R{\mathbb R}

\def\cH{\mathcal H}
\def\cK{\mathcal K}

\def\cS{\mathcal S}

\def\scrB{\mathscr B} 
\def\scrC{\mathscr C} 
\def\scrD{\mathscr D} 
\def\scrE{\mathscr E} 
\def\scrL{\mathscr L} 
\def\scrM{\mathscr M} 

\def\re{{\rm Re}}

\def\Int{{\rm int\,}}

\numberwithin{equation}{section}

\begin{document}

\date{\today}
\thanks{N. Roidos is participating in a research project which is implemented in the framework of H.F.R.I call ``Basic research Financing (Horizontal support of all Sciences)'' under the National Recovery and Resilience Plan ``Greece 2.0" funded by the European Union - NextGenerationEU (H.F.R.I. Project Number: 14758).}

\thanks{N. Roidos and E. Schrohe were supported by Deutsche Forschungsgemeinschaft, Grant \mbox{SCHR 319/9-1}}

\title[Asymptotics of the PME near Conical Singularities]{Asymptotics of Solutions to the Porous Medium Equation near Conical Singularities}

\author{Nikolaos Roidos}
\address{Nikolaos Roidos. Department of Mathematics, University of Patras, 26504 Rio Patras, Greece}
\email{roidos@math.upatras.gr}

\author{Elmar Schrohe} 
\address{E. Schrohe. Leibniz University Hannover, Institute of Analysis, Welfengarten 1, 30167 Hannover, Germany }
\email{schrohe@math.uni-hannover.de} 

\maketitle

\begin{abstract}We show that, on a manifold with conical singularities, the asymptotics of the solutions to the porous medium equation near the conical points are determined by the spectrum of the Laplacian on the cross-section of the cone. The key to this result is a precise description of the maximal domain of the cone Laplacian.
\end{abstract} 

\tableofcontents

\section{Introduction and Main Results} 
In this article, we study the porous medium equation (PME)
\begin{eqnarray}
\label{PME}
\partial_t u -\Delta (u^m) &=& F(t,u) \\
\label{PME2}
u_{|t=0}&=& u_0
\end{eqnarray}
on a manifold with conical singularities. Among other phenomena, the PME models the flow of a gas in a porous medium. In the above equation, $u$ is the density of the gas, $t$ is a time parameter, $\Delta$ is the Laplace-Beltrami operator, $m>0$, and $F$ is a forcing term. We assume the initial value $u_0$ to be strictly positive.
We showed in \cite{RS2} that the PME on a manifold with conical singularities has a unique maximal regularity solution in weighted cone Sobolev spaces for strictly positive initial data. For $F=0$, long time existence of these solutions was established in \cite{RS3}. The approach, based on a theorem of Cl\'ement and Li \cite{CL}, made the geometry of the conical singularities partly visible. Namely, the best possible choice of the weight in the weighted Sobolev space making up the domain of the Laplacian is linked to the first nonzero eigenvalue of the associated Laplacian on the cross-sections of the cone, see Theorem \ref{old}, below.

In this article, we will introduce a refined setting that allows us to determine the asymptotics of the solution $u$ near the tip. This enables us to answer a basic question: 
\begin{quote}
``Suppose we start with an initial value of the form $u_0=c_0+v_0$, where $c_0$ is a positive constant and $v_0\ge0$ vanishes to infinite order near the tip. What are the asymptotics the solution can develop in short time?'' 
\end{quote} 

To this end we will use a different approach that relies on an alternative form of the porous medium equation: The substitution $v=u^m$ transforms \eqref{PME} into the equation 
\begin{eqnarray}
\label{PMEAlt}
\partial_t v -mv^{(m-1)/m} = G(t,v), \quad v(0) = v_0, 
\end{eqnarray}
with $G(t,v) =mv^{(m-1)/m}F(t, v^{1/m})$ and $v_0=u_0^m$. 

We will first determine the asymptotics of $v$ and derive from these the asymptotics of $u$. The key is a spectral invariance result, see Theorem \ref{SI2}, below. 

The terms in the expansion of are polynomials in functions which, near the tip, are of the form $x^{-q_j^-}c_j(y)$ and $x^{-q_j^-}\ln x c_j(y)$, where $x$ is the distance to the tip, $q_j^-\le0$ is determined by the $j$-th eigenvalue $\lambda_j$ of the Laplacian on the cross-section, see \eqref{qj}, $y$ is a local variable in the cross-section, and $c_j$ belongs to the associated $\lambda_j$-eigenspace. 

We will start by recalling basic elements of the calculus on conic manifolds; see also \cite{S24} for more details. Readers familiar with these issues may proceed immediately to Section \ref{mainres}. 

This article continues the research on nonlinear parabolic evolution equations on manifolds with conical singularities by S. Coriasco, P. Lopes, J. Seiler, Y. Shao and the authors (\cite{CSS1}, \cite{LR2}, \cite{LR1}, \cite{Ro1}, \cite{RS4}, \cite{RSS}, \cite{RSh1}, \cite{RSh2}); see also \cite{BV},
\cite{BBGM24}, \cite{GHUV}, \cite{MP73}, 
\cite{MRS15}, \cite{RoSa} to mention just a few.

\subsection{Manifolds with conical singularities}
We model a manifold with conical singularities by an $(n+1)$-dimensional manifold $\B$, $n\ge1$, with boundary $\partial \B$ together with a degenerate Riemannian metric. More precisely, in a collar neighborhood $[0,1)\times \partial \B$ of the boundary, we fix coordinates $(x,y)$, where $x$ is a boundary defining function and $y$ a coordinate along the boundary. We then endow the interior $\Int(\B)$ of $\B$ with a Riemannian metric which, in the above collar neighborhood, assumes the form 
\begin{eqnarray}
\label{metric}
g= dx^2 + x^2 h(x),
\end{eqnarray}
where $x\mapsto h(x)$ is a smooth (up to $x=0$) family of non-degenerate Riemannian metrics on $\partial \B$. 
We speak of a straight conical singularity when $h$ is actually independent of $x$. 
From this perspective we can view $\partial \B$ as the cross-section of the cone. Note that the boundary may have several components corresponding to several conical singularities. 

\subsection{The Laplacian on Mellin Sobolev spaces}\label{laplacian}
A short computation shows that, in the above collar neighborhood of the boundary, the Laplace-Beltrami operator $\Delta$ with respect to the metric $g$ on $\Int(\B)$ can be written in the form 
\begin{eqnarray}
\label{LB}
\Delta = x^{-2} \left((-x\partial_x)^2-(n-1+H(x))(-x\partial_x) + \Delta_{h(x)}\right), 
\end{eqnarray}
where $H(x) = \frac12 \frac{x\partial_x\det h(x)}{\det h(x)}$, and $\Delta_{h(x)} $ is the Laplace-Beltrami operator on $\partial \B$ with respect to the metric $h(x)$. Note that $H(x)\to0$ as $x\to 0$ since $\det h$ is a smooth function of $x$ up to $x=0$ and that $H\equiv0$, if the cone is straight, i.e. the metric $h$ does not depend on $x$ near $\partial \B$. 

The Laplace-Beltrami operator (or for short the Laplacian) naturally acts on scales of weighted Mellin (or cone) Sobolev spaces $\cH^{s,\gamma}_p(\B)$, where $s,\gamma\in \R$ and $1<p<\infty$. They are easiest described when $s\in \N_0$. Then 
\begin{eqnarray}\label{hsg}
\lefteqn{\cH^{s,\gamma}_p(\B)= \Big\{ u\in H^s_{p,loc}(\Int(\B)): }\nonumber\\
&&x^{\frac{n+1}2-\gamma} \omega(x) (x\partial_x)^kD^\alpha_y u(x,y)\in L^p\Big([0,1)\times \partial \B; \frac{dxdy}x\Big), \forall \; k+|\alpha|\le s \Big\}. 
\end{eqnarray}
Here $\omega=\omega(x)$ is a cut-off function near the boundary, i.e., $0\le \omega\le 1$, $\omega\equiv 1$ near $x=0$ and $\omega\equiv 0$ near $x=1$. Obviously, the space $\cH^{s,\gamma}_p(\B)$ is independent of the choice of $\omega$ up to equivalent norms. For $s=0$ and $p=2$ this furnishes the $L^2$-space with respect to the metric \eqref{metric} up to an equivalent norm. See the Appendix for details. 

\subsection{Closed extensions} Clearly, the Laplacian is a bounded operator 
\begin{eqnarray*}
\Delta: \cH^{s+2,\gamma +2}_p(\B) \to \cH^{s,\gamma}_p(\B)
\end{eqnarray*}
for all choices of $s,\gamma$ and $p$. The maximal regularity approach requires us to consider it as a closed unbounded operator between Banach spaces. At first glance, one might be inclined to choose $\cH^{s,\gamma}_p(\B)$ as the space in which the Laplacian acts and $\cH^{s+2,\gamma+2}_p(\B)$ as its domain. However, this might not be a closed extension, see Theorem \ref{Dom}, below. Moreover, supposing that $s>(n+1)/p$, the functions in $\cH^{s+2,\gamma+2}_p(\B)$ are continuous. As $x\to0^+$, they will tend to zero, if $\gamma+2\ge (n+1)/2$, and they may be unbounded if $\gamma+2< (n+1)/2$. One certainly wants functions in the domain that can attain nonzero values as $x\to 0$, on the other hand one would not want functions blowing up at $x=0$. 
The possible closed extensions of the Laplacian as an unbounded operator in $\cH^{s,\gamma}_p(\B)$, can be determined as laid out in \cite{SS2}, building on work of Gil, Krainer and Mendoza \cite{GKM1}, \cite{GKM2} and Lesch \cite{Lesch}. A crucial role is played by the poles of the inverted principal Mellin symbol. The principal Mellin symbol $\sigma_M(\Delta)$ of the Laplacian is the operator-valued polynomial 
\begin{eqnarray}
\label{sigmaM}
\sigma_M(\Delta): \C \to \scrL(H^2(\partial \B), L^2(\partial \B))
\end{eqnarray}
given by 
\begin{eqnarray}
\label{conormal}
\sigma_M(\Delta)(z) = z^2 -(n-1)z +\Delta_{h(0)}.
\end{eqnarray}
Clearly, the points of non-invertibility are
\begin{eqnarray}
\label{qj}
q_{j}^{\pm}=\frac{n-1}{2}\pm\sqrt{\Big(\frac{n-1}{2}\Big)^{2}-\lambda_{j}},\ j=0,1,2, \ldots,
\end{eqnarray}
where $0=\lambda_0>\lambda_1> \lambda_2>\ldots$ are the different eigenvalues of $\Delta_{h(0)}$.
We conclude that 
\begin{eqnarray}\label{inverse}
\sigma_M(\Delta)^{-1}(z) = \sum_{j=0}^\infty \frac{\pi_j}{(z-q_j^+)(z-q_j^-)},
\end{eqnarray}
where $\pi_j$ is the orthogonal projection in $L^2(\partial \B)$ onto the eigenspace $E_j$ associated with the eigenvalue $\lambda_j$. This shows that the poles of $z\mapsto \sigma_M(z)^{-1}$ are all simple except when $n=1$, for then $z=q_0^+=q_0^-=0$ is a double pole. 

The Laplacian, as an unbounded operator in $\cH^{s,\gamma}_p(\B)$, has two special closed extensions: the minimal, $\Delta_{\min}$, which is the closure of $\Delta$ with domain $C^\infty_c(\Int(\B))$ and the maximal $\Delta_{\max} $ whose domain consists of all $u\in \cH^{s,\gamma}_p(\B)$ such that $\Delta u\in \cH^{s,\gamma}_p(\B)$. 

\begin{theorem}\label{Dom}
Assume that $\frac{n+1}2-\gamma-2$ is not a pole of $\sigma_M(\Delta)^{-1}$. Then
\begin{eqnarray*}
\scrD(\Delta_{\min}) = \cH^{s+2,\gamma+2}_p(\B).
\end{eqnarray*}
The domain of the maximal extension is
\begin{eqnarray}\label{Dmax}
\scrD(\Delta_{\max}) = 
\cH^{s+2,\gamma+2}_p(\B) \oplus \bigoplus_{q_j^\pm\in I_\gamma} \scrE_{q_j^\pm},
\end{eqnarray}
where the sum is over all $q_j^\pm$ in the interval 
\begin{eqnarray}
\label{Igamma}
I_\gamma=\Big(\frac{n+1}2-\gamma-2, \frac{n+1}2-\gamma \Big) ,
\end{eqnarray}
and the $\scrE_{q_j^\pm}$ are finite-dimensional spaces of smooth functions on $\Int(\B)$ with special asymptotics as $x\to 0$ that can be determined explicitly, see the Appendix. 

As a consequence, any closed extension $\underline \Delta$ of the Laplacian has a domain of the form
$$\scrD(\underline\Delta) = \cH^{s+2,\gamma+2}_p(\B) \oplus \underline\scrE,$$
with a subspace $\underline\scrE$ of $\bigoplus_{q_j^\pm\in I_\gamma} \scrE_{q_j^\pm}$, provided $\frac{n+1}2-\gamma-2$ is not a pole of $\sigma_M(\Delta)^{-1}$.
\end{theorem} 

\begin{remark}
In case $\frac{n+1}2-\gamma-2$ {\em is} a pole of $\sigma_M(\Delta)^{-1}$, the minimal domain is 
$$\scrD(\Delta_{\min}) = \Big\{ u\in \bigcap_{\varepsilon>0} \cH^{s+2,\gamma+2-\varepsilon}_p(\B): 
\Delta u \in \cH^{s,\gamma}_p(\B)\Big\}.$$
\end{remark} 

\subsection{Previous work}\label{Previous}In the articles \cite{RS2} and \cite{RS3} we worked with the extensions $\underline \Delta$ of the Laplacian with the domain
\begin{eqnarray}\label{dom_old}
\scrD(\underline\Delta) = \cH^{s+2,\gamma+2}_p(\B) \oplus \underline\scrE_0,
\end{eqnarray}
where $\gamma$ was chosen such that $\max\{-2, q_1^-\}<\frac{n+1}2-\gamma-2 <0$. The interval $I_\gamma$ defined in \eqref{Igamma} therefore contains $q_0^-=0$ and possibly some of the $q_j^+$, $j\ge1$, but none of the $q_j^-$ for $j\ge1$. The cone was assumed to be straight, and the space $\underline\scrE_0$ therefore consisted of functions locally constant near $\partial \B$, 
\begin{eqnarray}
\label{scrE0}
\underline\scrE_0 =\{ u\in C^\infty(\B): u(x,y) = \omega(x) e(y); e\in E_0\}.
\end{eqnarray}

The following is Theorem 1.1 in \cite{RS2}:

\begin{theorem}\label{old} 
For $p$ and $q$ sufficiently large and a strictly positive initial value $u_0$ in the real interpolation space $(\scrD(\underline \Delta), \cH^{s,\gamma}_p(\B))_{1/q,q}$, the PME \eqref{PME} has a unique solution 
$$u\in L^q(0,T,\scrD(\underline\Delta)) \cap W^1_q(0,T, \cH^{s,\gamma}_p(\B))\hookrightarrow
C([0,T], (\cH^{s,\gamma}_p(\B),\scrD(\underline \Delta))_{1-1/q,q})$$
for some $T>0$, where the last embedding follows from \cite[Theorem III.4.10.2]{Amann}.
\end{theorem}

It follows from \cite[Lemma 5.2]{RS1} (an independent proof will be given, below) that, for every $\varepsilon >0$, 
\begin{eqnarray}\label{embedding1}
\cH^{s+2-2/q+\varepsilon, \gamma+2-2/q+\varepsilon}_p(\B)+\underline{\scrE}_0& \hookrightarrow &
(\cH^{s,\gamma}_p(\B),\scrD(\underline \Delta))_{1-1/q,q},\\
&\hookrightarrow&
\cH^{s+2-2/q-\varepsilon, \gamma+2-2/q-\varepsilon}_p(\B)+\underline\scrE_0,\nonumber
\end{eqnarray}
where the sum is direct whenever $\gamma+2-2/q-\varepsilon>(n+1)/2$. 

In case $q_1^-\le-2$, this is an optimal result. We can choose $\gamma$ so that $\frac{n+1}2-\gamma-2$ is only slightly larger than $-2$ and conclude that the non-constant part of any solution to \eqref{PME} with $\scrD(\underline \Delta)$ given by \eqref{dom_old} belongs to $\cH_p^{s+2-2/q-\varepsilon,\gamma+2-2/q-\varepsilon}(\B)$ for any $\varepsilon>0$. for large $q$, this is almost two orders flatter than the constant part. 

\subsection{Main Results}\label{mainres}
Define $k$ as the largest index such that
\begin{eqnarray*}
q_{k+1}^-\le-2<q_k^-< \ldots <q_1^-<q_0^-=0.
\end{eqnarray*}
We shall assume that $q_1^->-2$, for else we will be back in the situation of Section \ref{Previous}. 
Choose $\gamma$ such that 
\begin{eqnarray}\label{gamma_new}
-2<\frac{n+1}2-\gamma-2 <q_k^-
\end{eqnarray}
and that $(n+1)/2-\gamma$ is not a pole of $\sigma_M(\Delta)^{-1}$.
The interval $I_\gamma$ defined in \eqref{Igamma} then contains $q_0^-,\ldots, q_k^-$ and possibly some of the $q_j^+$, $j\ge0$. Since $(n+1)/2-\gamma-2$ is not a pole of $\sigma_M(\Delta)^{-1}$, Theorem \ref{Dom} is applicable. 

\subsubsection{The choice of the closed extension}
We consider $\Delta$ as an unbounded operator in $\cH^{s,\gamma}_p(\B)$ for some $s\in\R$ and $1<p<\infty$ to be determined later on. We fix the extension $\underline \Delta$ of $\Delta$ with the domain
\begin{eqnarray}\label{domD}
\scrD(\underline\Delta) = \cH^{s+2,\gamma+2}_p(\B)\oplus \bigoplus_{j=1}^k \scrE_{q_j^-} \oplus \underline\scrE_0 . 
\end{eqnarray}
The asymptotics spaces $\scrE_{q_j^-}$, $j=0,\ldots, k$, are given explicitly in Lemmas \ref{Esigma}, \ref{nge2pole} and \ref{n1gammapos} in the Appendix; $\underline\scrE_0$ is as in \eqref{scrE0}. 
Then $\underline \scrE_0=\scrE_0$, if $n\ge2$, while $\underline \scrE_0$ is a proper subspace of $\scrE_0$ for $n=1$.

We will establish the following result: 

\begin{theorem}\label{main} 
Fix $\gamma$ as in \eqref{gamma_new}, and choose $1<p,q<\infty$ so large that 
\begin{eqnarray}\label{pq}
\frac{n+1}{p} + \frac2q<1\quad \text{ and }\quad \frac{n+1}2-\gamma-2 +\frac4q<0.
\end{eqnarray}
Moreover, choose $\scrD(\underline\Delta)$ as in \eqref{domD} with 
\begin{eqnarray}\label{s}
s>-1+\frac{n+1}p+\frac 2q. 
\end{eqnarray}
For any initial value $v_0$ in the interpolation space $(\cH^{s,\gamma}_p(\B), \scrD(\underline\Delta))_{1-1/q,q}$ that is strictly positive on $\B$, the porous medium equation \eqref{PMEAlt} 
$$ \partial_t v-mv^{(m-1)/m}\Delta v = G(t,v); \quad v(0)=v_0$$ 
with $v_0 \in (\cH^{s,\gamma}_p(\B), \scrD(\underline\Delta))_{1-1/q,q}$ and 
forcing term 
$$G\in C^{1-,1-}([0,T_0]\times U, \cH^{s,\gamma}_p(\B)),$$
where $T_0>0$ and $U$ is an open neighborhood of $v_0$ in 
$(\cH^{s,\gamma}_p(\B), \scrD(\underline\Delta))_{1-1/q,q}$, has a unique solution 
\begin{eqnarray}\label{solution}
v\in W^{1,q}(0,T; \cH^{s,\gamma}_p(\B)) \cap L^q(0,T, \scrD(\underline \Delta))
\end{eqnarray}
for suitable $0<T\le T_0$.
\end{theorem} 
Recall that $n+1=\dim(\B)$. Since $\frac{n+1}2-\gamma-2<0$ by \eqref{gamma_new}, condition \eqref{pq} on $q$ can always be fulfilled. 
Moreover, we have the following extension of \eqref{embedding1}:

\begin{lemma}\label{interpolation}Let $\gamma$ be as in \eqref{gamma_new}. 
For every $\varepsilon>0$ we have continuous and dense embeddings
\begin{eqnarray}\label{embedding2}
\cH^{s+2-2/q+\varepsilon, \gamma+2-2/q+\varepsilon}_p(\B) &+&\bigoplus_{j=1}^k \scrE_{q_j^-}+\underline{\scrE}_0 \hookrightarrow 
(\cH^{s,\gamma}_p(\B), \scrD(\underline\Delta))_{1-1/q,q}\nonumber\\
&\hookrightarrow&
\cH^{s+2-2/q-\varepsilon, \gamma+2-2/q-\varepsilon}_p(\B)+ \bigoplus_{j=1}^k \scrE_{q_j^-}\oplus \underline\scrE_0.
\end{eqnarray}

{\rm(i)} The sum on the right hand side is direct when $\frac{n+1}2-\gamma-2+\frac2q +\varepsilon <q_k^-$, which can be achieved in view of \eqref{gamma_new} by taking $q$ large and $\varepsilon$ small. 

{\rm(ii)} For general $q$ we find an index $0\le r\le k$ such that 
$$\max\{-2,q_{r+1}^-\} <\frac{n+1}2-\gamma-2+\frac2q+\varepsilon< q_r^-$$ for all sufficiently small $\varepsilon>0$. Then the right hand side is 
$$\cH^{s+2-2/q-\varepsilon, \gamma+2-2/q-\varepsilon}_p(\B)\oplus \bigoplus_{j=1}^r \scrE_{q_j^-}\oplus\underline\scrE_0.$$
\end{lemma} 
Together with the embedding \cite[Theorem III.4.10.2]{Amann} we see 

\begin{corollary}If $q$ is so large that $-2<(n+1)/2-\gamma-2+2/q <q_k^-$, any solution $v$ of the PME in Theorem \ref{main} will satisfy 
\begin{eqnarray*}
v\in C\Big([0,T], \cH^{s+2-2/q-\varepsilon, \gamma+2-2/q-\varepsilon}_p(\B)\oplus \bigoplus_{j=1}^k \scrE_{q_j^-}\oplus\underline\scrE_0\Big)
\end{eqnarray*} 
for any $\varepsilon>0$ sufficiently small. 

For general $q$ we obtain partial asymptotics in the sense of Lemma {\rm \ref{interpolation}\,(ii)}.

\end{corollary}

Recall, however, that $v$ in Theorem \ref{main} is the solution to the alternative form of the porous medium equation \eqref{PMEAlt} and that we are actually interested in the asymptotics of $u=v^{1/m}$ solving \eqref{PME}. In order to understand these we introduce a new Banach algebra adapted to the situation in Lemma \ref{interpolation} (ii).

\begin{definition}
\label{calB}
For $1<p<\infty$ let $\underline s>(n+1)/p$, $\underline \gamma>(n+1)/2$ and $r\in \{0,\ldots,k\}$
such that 
\begin{eqnarray*}
\max\{-2,q_{r+1}^-\} <\frac{n+1}2-\underline \gamma< q_r^-.
\end{eqnarray*}
Recall that the elements of $\cH^{\underline s, \underline\gamma}_p(\B)$ as well as those of the spaces $\scrE_{q_j^-}$, $1\le j\le r$, and those of $\underline\scrE_0$ are continuous on $\B$. 
Denote by $\scrB$ the subalgebra of $C(\B)$ generated by the functions in $\cH^{\underline s, \underline\gamma}_p(\B)$, $\scrE_{q_j^-}$, $1\le j\le r$, and $\underline\scrE_0$.
\end{definition}

\begin{theorem}\label{SI2}
$\scrB$ is a spectrally invariant Banach subalgebra of $C(\B)$. 
In particular, $\scrB$ is closed under holomorphic functional calculus. The quotient $\scrB/\cH^{\underline s,\underline\gamma}_p(\B)$ is finite dimensional. 
\end{theorem} 
Spectral invariance means that whenever a function $w\in \scrB$ has an inverse in $C(\B)$, the inverse belongs to $\scrB$. The closedness under holomorphic functional calculus follows from the continuity of inversion in Banach algebras and Cauchy's integral formula. 
 
\section{Proofs}
We assume that we are in the setting outlined in Section \ref{mainres}, in particular, $\gamma$ is chosen as in \eqref{gamma_new}, and $\underline \Delta$ is the extension of the Laplacian with domain \eqref{domD}.
\subsection{Proof of Lemma \ref{interpolation}}
Let $s\in \R$, $1<p,q<\infty$. 
According to \cite[Lemma 5.4]{CSS1}, the embeddings
\begin{eqnarray*}
\cH^{s+2-2/q+\varepsilon, \gamma+2-2/q+\varepsilon}_p(\B) \hookrightarrow (\cH^{s,\gamma}_p(\B), \cH^{s+2,\gamma+2}_p(\B))_{1-1/q,q}\hookrightarrow \cH^{s+2-2/q-\varepsilon, \gamma+2-2/q-\varepsilon}_p(\B) 
\end{eqnarray*}
are continuous for all $\varepsilon>0$; they have dense range, since $\cH^{s+2, \gamma+2}_p(\B)$ is dense in all spaces. Moreover, the spaces $\scrE_{q_j^-}$, $j=1, \ldots, k$, and $\underline \scrE_0$ are contained in $\cH_{p}^{s,\gamma}(\B)\cap \scrD(\underline\Delta)$. Hence we have a continuous and dense embedding 
 \begin{eqnarray*}
\cH^{s+2-2/q+\varepsilon, \gamma+2-2/q+\varepsilon}_p(\B)+ \bigoplus_{1\le j\le k} \scrE_{q_j^-}
+ \underline \scrE_0 \hookrightarrow (\cH^{s,\gamma}_p(\B), \scrD(\underline\Delta))_{1-1/q,q}.
\end{eqnarray*}
To see the converse direction, we first note that $\scrD(\underline\Delta)\subseteq \scrD(\Delta_{\max})$ as an unbounded operator in $\cH^{s,\gamma}_p(\B)$, and therefore $\Delta(\scrD(\underline\Delta)) \subseteq \cH^{s,\gamma}_p(\B)$. We conclude that 
\begin{eqnarray*}
\Delta((\cH^{s,\gamma}_p(\B), \scrD(\underline\Delta))_{1-1/q,q}) 
\hookrightarrow (\cH^{s-2,\gamma-2}_p(\B), \cH^{s,\gamma}_p(\B))_{1-1/q,q}\hookrightarrow 
\cH^{s-2/q-\varepsilon,\gamma-2/q-\varepsilon}_p(\B).
\end{eqnarray*}
Moreover, 
$$
(\cH^{s,\gamma}_p(\B), \scrD(\underline\Delta))_{1-1/q,q}\hookrightarrow \cH^{s,\gamma}_p(\B) \hookrightarrow \cH^{s-2/q-\varepsilon,\gamma-2/q-\varepsilon}_p(\B).
$$
Hence $(\cH^{s,\gamma}_p(\B), \scrD(\underline\Delta))_{1-1/q,q}$ is a subset of the maximal domain of $\Delta$, considered as an unbounded operator in $\cH^{s-2/q-\varepsilon,\gamma-2/q-\varepsilon}_p(\B)$.
Assuming that $\frac{n+1}2-\gamma+\frac2q+\varepsilon$ is not a pole of $\sigma_M(\Delta)^{-1}$ (otherwise change $\varepsilon$ slightly), this maximal domain is 
\begin{eqnarray*}
\cH^{s+2-2/q-\varepsilon,\gamma+2-2/q-\varepsilon}_p(\B)\oplus \bigoplus_q \scrE_q,
\end{eqnarray*}
where the direct sum is over all $q_j^\pm$ in the interval $I_{\gamma-2/q-\varepsilon}$. 
On the other hand we know from interpolation theory that $\scrD(\underline\Delta)$ is dense in the interpolation space 
$(\scrD(\underline\Delta),\cH^{s,\gamma}_p(\B))_{1/q,q}$. Hence 
\begin{eqnarray*}
(\cH^{s,\gamma}_p(\B), \scrD(\underline\Delta))_{1-1/q,q}\hookrightarrow
\cH^{s+2-2/q-\varepsilon,\gamma+2-2/q-\varepsilon}_p(\B)
+ \bigoplus\scrE_{q_j^-} + \underline \scrE_0,
\end{eqnarray*}
where the direct sum is over those $q_j^-$, $j=1,\ldots,k$, that lie in the interval $I_{\gamma-2/q-\varepsilon}$.

To see the directness of the sum we infer from \eqref{hsg} that the space $\scrE_{q_j^-}$ is contained in $\cH^{\infty,\sigma}_p(\B)$ if and only if 
\begin{eqnarray}
\label{qjinhsg}
q_j^-<\frac{n+1}2-\sigma.
\end{eqnarray}

\subsection{Outline of the proof of the main theorem.}
An important tool in the proof of Theorem \ref{main} will be the following result by Cl\'ement and Li \cite{CL}: 

\begin{theorem}\label{CL} Let $X_1\hookrightarrow X_0$ be Banach spaces, $X_1$ dense in $X_0$, $1<q<\infty$, $T_0>0$. 
In $L^{q}(0,T_{0};X_{0})$ consider the quasilinear parabolic problem 
\begin{gather}\label{QL}
\partial_tv(t)+A(v(t))v(t)=f(t,v(t))+g(t), \quad t\in(0,T_{0}),\quad u(0)=u_{0}.
\end{gather}

Assume that there exists an open neighborhood $U$ of $v_0$ in the real interpolation space $(X_0,X_1)_{1-1/q,q}$ such that $A(v_0): X_{1}\rightarrow X_{0}$ has maximal $L^{q}$-regularity and 
\begin{itemize}
\item[(H1)] $A\in C^{1-}(U, \scrL(X_1,X_0))$, 
\item[(H2)] $f\in C^{1-,1-}([0,T_0]\times U, X_0)$,
\item[(H3)] $g\in L^q(0,T_0; X_0)$.
\end{itemize}
Then there exist a $T>0$ and a unique solution $v\in W^1_q(0,T;X_0) \cap L^q(0,T;X_1)$ to Equation \eqref{QL} on $(0,T)$.
In particular, $u\in C([0,T];(X_0,X_1)_{1-1/q,q})$ by \cite[Theorem III.4.10.2]{Amann}. 
\end{theorem}

This applies to our problem \eqref{PMEAlt} with the choices $X_0 = \cH^{s,\gamma}_p(\B)$, $X_1= \scrD(\underline\Delta)$ and $A(v) = -mv^{(m-1)/m} \underline \Delta$ together with the data given in Theorem \ref{main}.

Of central importance in Theorem \ref{CL} is the maximal regularity of $A(v_0)$; see Definition \ref{MaxReg}. In order to establish it, we will first show in the following Section \ref{HIDelta} that $c-\underline\Delta$ has a bounded $H^\infty$-calculus in $\cH^{s,\gamma}_p(\B)$ for every $c>0$ with respect to any sector 
\begin{eqnarray}\label{Lambda} 
\Lambda =\Lambda_\theta = \{ z=re^{i\phi} \in \C:r\ge 0; \theta \le\phi\le2\pi-\theta\}. 
\end{eqnarray}
where $0<\theta<\pi$. 

From this we will derive that $c -mv^{(m-1)/m}\underline\Delta$, is $R$-sectorial for the same sector $\Lambda$ for every $v$ in a neighborhood $U$ of the strictly positive initial value $u_0$, provided $c>0$ is sufficiently large. Since the angle $\theta$ can be chosen smaller than $\pi/2$, a theorem of Weis \cite[Theorem 4.2]{W} finally shows that $A(v)$ has maximal regularity for all $v\in U$, in particular for $v_0$.

In Section \ref{H1H2} we will then verify that the conditions (H1) and (H2) are fulfilled; (H3) is superfluous in our case. 

\subsection{Bounded $H^\infty$-calculus for $c-\underline \Delta$, $c>0$}\label{HIDelta}

We will show: 

\begin{theorem}\label{hinftyDelta}
Let $c>0$, $s\in \R$, $1<p<\infty$ and $\gamma$ be chosen according to \eqref{gamma_new}. Then 
$c-\underline\Delta$, considered as an unbounded operator in $\cH^{s,\gamma}_p(\B)$ with the domain \eqref{domD}, has a bounded $H^\infty$-calculus on $\Lambda$. 
\end{theorem} 

The proof relies on work in \cite{SS1} and \cite{SS2}. In order to state the details, we will need two more definitions. 

\begin{definition} 
The model cone operator $\widehat \Delta$ associated with the Laplacian $\Delta$ is the operator obtained by evaluating the coefficients at $x=0$, i.e. 
\begin{eqnarray}
\label{modelcone}
\widehat \Delta = x^{-2}((-x\partial_x)^2 - (n-1)(-x\partial_x) +\Delta_{h(0)}).
\end{eqnarray}
\end{definition}

The name originates from the fact that this operator models the behavior of the operator $\Delta$ at the tip of the cone. It therefore can be considered as an operator on the infinite cone with cross-section $\partial \B$. Correspondingly, it acts on a scale of Sobolev spaces associated with this cone:

\begin{definition}
For $s,\gamma\in\R$ and $1<p<\infty$ denote by $\cK^{s,\gamma}_p(\R_+\times\partial \B)$ the Banach space of all distributions $u$ on $\R_+\times\partial \B$ such that for every cut-off function $\omega$\begin{enumerate}\renewcommand\labelenumi{\rm(\roman{enumi})}
\item $\omega u \in \cH^{s,\gamma}_p(\B)$, and 
\item given a coordinate map $\kappa:U\subseteq\partial \B\to \R^n$ and $\phi\in C^\infty_c(U)$, the push forward $\chi_*((1-\omega)(x)\phi(y) u)$ is an element of $H^{s}_p(\R^{1+n})$, where $\chi(x,y) = (x,x\kappa(y))$. 
\end{enumerate}
\end{definition} 

Away from the tip, $\cK^{s,\gamma}_p(\R_+\times \partial \B)$ is the canonical Sobolev space $H^{s}_p$ on the outgoing cone with cross-section $\partial \B$, defined by considering $x\in (0,\infty)$ as a fixed coordinate. 
For $p=2$, these spaces were introduced in \cite[Section 2.1.1]{Schu}; see also \cite[Section 4.2]{ScSc95}. One can show: 

\begin{lemma}
For all $s,\gamma\in \R$ and $1<p<\infty$, the model cone operator $\widehat \Delta$ induces a bounded linear map
$$\widehat\Delta: \cK^{s+2,\gamma+2}_p(\R_+\times\partial \B)\to \cK^{s,\gamma}_p( \R_+\times\partial \B).$$ 
\end{lemma} 

\subsubsection*{Proof of Theorem \ref{hinftyDelta}} 
It follows from \cite[Theorem 5.2]{SS2} that, given a closed extension $\underline A$ of a general cone differential operator $A$ and a sufficiently large $c>0$, the operator $c-\underline A$ has a bounded $H^\infty$-calculus on $\Lambda$ as an unbounded operator in $\cH^{s,\gamma}_p(\B)$, $s\ge0$, provided it is parameter-elliptic with respect to $\Lambda$. Parameter-ellipticity is defined by the properties (E1), (E2) and (E3) stated at the beginning of \cite[Section 4]{SS2}. We will reproduce and verify them for the case of the operator $\underline \Delta$. 
		
\begin{itemize}
\item {\bf(E1)} requires that both the principal pseudodifferential symbol $\sigma^2_\psi(\Delta)$ of $\Delta$ and the rescaled pseudodifferential symbol $\widetilde\sigma^2_\psi(\Delta)$, which is defined in the collar neighborhood of the boundary, have no eigenvalues in $\Lambda$. 

In the present case, we have $\sigma^2_\psi(\Delta)(z,\zeta) = -|\zeta|^2_{g}$ for $(z,\zeta)\in T^*(\Int(\B))\setminus \{0\}$, where $g$ is the Riemannian metric on $\Int(\B)$ and $|\cdot|_g$ the induced norm on $T^*(\Int(\B))$. Near $\partial B$, for $(x,y,\xi,\eta)$ in $T^*(\R_+\times \partial \B)\setminus \{0\}$, 
$$\widetilde\sigma^2_\psi(\Delta)(x,y,\xi,\eta) := x^2\sigma^2_\psi(\Delta)(x,y,\xi/x, \eta) = -\xi^2+\sigma^2_\psi(\Delta_{h(x)}) = -\xi^2-|\eta|_{h(x)}^2,$$
with the metric $h(x)$ on $\{x\}\times \partial \B$. Neither symbol has eigenvalues in $\Lambda$.

\item {\bf(E2)} requires that $\sigma_M(\Delta)(z)$ be invertible in $z=(n+1)/2-\gamma$ and $z= (n+1)/2-\gamma-2$. This is the case with the choice of $\gamma$ in \eqref{gamma_new}. 
\item {\bf(E3)} requires that $\|\lambda (\lambda-\widehat{\underline\Delta})^{-1}\|_{\cK^{0,\gamma}_2(\R_+\times \partial \B)}$ is bounded for $\lambda \in \Lambda$.
This will follow from our choice of $\scrD(\underline\Delta)$ and an application of \cite[Theorem 6.5]{SS2}. Here are the details:
\end{itemize} 

According to Theorem \ref{max_domain1} and Lemmas \ref{Esigma} and \ref{nge2pole} in the Appendix we have 
\begin{eqnarray}
\label{domhatdelta}
\scrD(\underline{\widehat\Delta})
= \cK^{s+2,\gamma+2}_p(\R_+\times\partial \B)\oplus 
\bigoplus_{j=1}^k\widehat\scrE_{q_j} \oplus 
{\underline {\widehat\scrE}_0}
\end{eqnarray}
with
\begin{eqnarray*}
\widehat \scrE_{q_j^-}
&=& \{ u\in C^\infty(\R_+\times \partial \B): u(x,y) = \omega(x) x^{-q_j^-} e(y); e \in E_j \} \text{ and }\\
{\underline {\widehat\scrE}_0} &=& \{u\in C^\infty(\R_+\times \partial \B): u(x,y) = \omega(x)e(y); e\in E_0\}.\end{eqnarray*}.

We shall now derive property (E3) from \cite[Theorem 6.5]{SS2}, which we recall for the convenience of the reader: 
\begin{theorem}\label{CondE3} 
Let $|\gamma|<(n+1)/2$ be chosen such that (E2) holds and let $\scrD(\widehat{\underline \Delta})$ 
have a domain of the form 
$$\cK^{s,\gamma}_p(\R_+\times \partial \B) \oplus \bigoplus_{q\in I_\gamma} \widehat{\underline\scrE}_q 
$$
with $I_\gamma$ as in \eqref{Igamma}, $q$ a pole of $\sigma_M(\Delta)^{-1}$ and subspaces $\widehat{\underline\scrE}_q\subseteq \widehat{\scrE}_q$ satisfying
\begin{enumerate}\renewcommand{\labelenumi}{{\rm (\roman{enumi}) }}
\item $\widehat{\underline\scrE}_{q}^\perp= \widehat{\scrE}_{n-1-q}$ for $q\in I_\gamma\cap I_{-\gamma} $
\item $\widehat{\underline\scrE}_{q} = \widehat{\scrE}_{q}$ for $q\in I_\gamma\setminus I_{-\gamma}$ and $\gamma\ge 0$
\item $\widehat{\underline\scrE}_{q} = \{0\}$ for $q\in I_\gamma\setminus I_{-\gamma}$ and $\gamma\le 0$.
\end{enumerate} 
Then $\Delta$ satisfies {\rm (E3)} for every sector $\Lambda\subset \C\setminus \R_+$.
\end{theorem}

Here, the spaces $\widehat{\underline \scrE}_q^\perp$ are defined as follows: If $q=q_j^\pm$ for some $j\ge1$ or if $n>1$, and 
$$\widehat{\underline \scrE}_{q_j^\pm} = \{ u : u(x,y) = \omega(x) x^{-q_j^\pm}e(y): e\in \underline E_j\}$$
for some subspace $\underline E_j \subseteq E_j$, then 
$$\widehat{\underline \scrE}_{q_j^\pm}^\perp = \{ u : u(x,y) = \omega(x) x^{-q_j^\mp}e(y): e\in \underline E_j^\perp\}.$$
In case $n=1$ and $j=0$, i.e., $q=q_0^+=q_0^-=0$, we let
$$\widehat{\underline \scrE}_0^\perp =
\begin{cases}
\widehat\scrE_0;& \widehat{\underline\scrE}_0=\{0\}\\
\{0\};&\widehat{\underline \scrE}_0=\widehat\scrE_0\\
\widehat{\underline\scrE}_0;&\widehat{\underline\scrE}_0=\{ u(x,y) = \omega(x)E_0\}.
\end{cases} 
$$
So let us verify (i), (ii) and (iii) in Theorem \ref{CondE3}: 
\subsubsection*{The case $\gamma\ge 0$} Here 
$$I_\gamma\cap I_{-\gamma} = \left]\frac{n+1}2+\gamma-2,\frac{n+1}2-\gamma\right[ \text{ and } 
I_\gamma\setminus I_{-\gamma} = \left] \frac{n+1}2-\gamma-2, \frac{n+1}2+\gamma-2\right[.
$$
According to \eqref{qj}, the $q_j^-$ lie to the left of $(n-1)/2$ on the real axis, the $q_j^+$ to the right at the same distance. Since $I_\gamma\cap I_{-\gamma}$ is symmetric about $(n-1)/2$, it will contain either both $q_j^-$ and $q_j^-$ or neither of them. 

In \eqref{gamma_new} we have chosen the full spaces $\widehat{\scrE}_{q_j^-}$ over the $q_j^-$ in $I_\gamma$ for $j=1,\ldots,k$, and also over $q_0^-$ in case $n>1$. Over any of the $q_j^+$, $j\ge1$, we have chosen the zero spaces.
For $n=1$ and $q_0^-=q_0^+=0$, the space $\widehat{\underline\scrE}_0$ is self-orthogonal by definition. Hence property (i) holds. Since $I\gamma\setminus I_{-\gamma} $ lies to the left of $I_\gamma\cap I_{-\gamma}$ on the real axis, the above consideration shows that it does not contain any of the $q_j^+$. Hence (ii) also holds. 

\subsubsection*{The case $\gamma\le 0$} Again, $I_\gamma\cap I_{-\gamma}$ is symmetric about $(n-1)/2$, so that we can argue as before. The interval 
$I_\gamma\setminus I_{-\gamma} = \left] \frac{n+1}2+\gamma-2, \frac{n+1}2-\gamma-2\right[$ lies to the right of $I_\gamma\cap I_{-\gamma}$ on the real axis, so it will not contain any of the $q_j^-$,
and condition (iii) holds. 

\subsubsection*{Conclusion for the case $s\ge0$} Since (i), (ii) and (iii) in Theorem \ref{CondE3} are fulfilled, property (E3) holds. Theorem \ref{hinftyDelta} now implies that $c-\underline\Delta$ has a bounded $H^\infty$-calculus as a closed unbounded operator in $\cH^{s,\gamma}_p(\B)$ for $s\ge0$ and for sufficiently large $c>0$. The $c$ has to be taken so large that $(c-\underline\Delta)^{-1}$ exists. In the case at hand, we know that the spectrum of $\underline\Delta$ is contained in $\R_{\le0}$ and contains $0$. Hence any positive $c$ will do.

\subsubsection*{The case $s<0$} We will show that the adjoint $c-(\underline\Delta)^*$ of $c-\underline\Delta$ with respect to the $\cH^{0,0}_2(\B)$ inner product has a bounded $H^\infty$-calculus in the dual space $\cH^{-s,-\gamma}_{p'}(\B)$, where $1/p+1/p'=1$.
As before, we will derive the existence of the bounded $H^\infty$-calculus from \cite[Theorem 5.2]{SS2} by checking the conditions (E1), (E2) and (E3) in Theorem \ref{CondE3}.

Since the Laplacian is formally self-adjoint, the principal symbols of $\Delta^*$ are those of $\Delta$. Hence (E1) and (E2) are fulfilled. 
Moreover, the model cone operators of $\Delta$ and $\Delta^*$ coincide except for their domains.

According to \cite[Theorem 5.3]{SS1}, the domain of the adjoint $\widehat{\underline\Delta}^*$ of a closed extension $\widehat{\underline \Delta}$ of $\widehat\Delta$ with domain 
$$\scrD(\widehat{\underline \Delta}) = \scrD(\widehat\Delta_{\min, \cK^{s,\gamma}_p(\R_+\times \partial \B)}) \oplus \bigoplus_{q\in I_\gamma} \widehat{\underline \scrE}_q$$
is 
$$\scrD(\widehat{\underline \Delta}^*) = \scrD(\widehat\Delta_{\min, \cK^{-s,-\gamma}_{p'}(\R_+\times\partial \B)}) \oplus \bigoplus_{q\in I_\gamma} \underline {\widehat\scrE}_q^\perp,$$
where the notation indicates in which space the minimal extension is taken. 

Since both the $q_j^\pm$ and the sets $I_\gamma \cup I_{-\gamma}$ and $I_\gamma\cap I_{-\gamma}$ lie symmetric about the point $(n-1)/2$, we have 
$$\frac{n+1}2+\gamma -2\not= q_j^\pm \text{ and } \frac{n+1}2 +\gamma\not= q_j^\pm$$
for all $j$ as consequence of our choice of $\gamma$, see \eqref{gamma_new}.
Hence 
$$\scrD(\widehat{\underline\Delta}^*)_{\min;\cK^{-s,-\gamma}_{p'} (\R_+\times \partial \B)} =
\cK^{-s+2,-\gamma+2}_p(\R_+\times\partial \B).$$
The fact that $\scrD(\widehat{\underline\Delta})$ satisfies (i), (ii) and (iii) of Theorem \ref{CondE3} then immediately implies that the same is true for $\scrD(\widehat{\underline\Delta}^*)$. 

\subsubsection*{Conclusion for $s<0$} Since the conditions in Theorem \ref{CondE3} are fulfilled, 
\cite[Theorem 5.2]{SS2} shows that $(c-\underline\Delta)^*$, considered as an unbounded operator in $\cH^{-s,-\gamma}_{p'}(\B)$ has a bounded $H^\infty$-calculus on $\Lambda$. 
Taking once more the formal adjoint, we conclude from \cite[2.11(v)]{DHP} that $c-(\underline\Delta)^{**} = c-\underline\Delta$ has a bounded $H^\infty$-calculus on $\Lambda$ in $\cH^{s,\gamma}_p(\B)$. 
The proof of Theorem \ref{hinftyDelta} is now complete. 
\hfill $\Box$

\subsection{Maximal regularity of $-mv^{(m-1)/m}\underline \Delta$}\label{SectMaxReg}

According to Lemma \ref{interpolation}\,(ii), the interpolation space 
$(\cH^{s,\gamma}_p(\B), \scrD(\underline\Delta))_{1-1/q,q}$ embeds into 
\begin{eqnarray}
\label{MRD}
\cH^{s+2-2/q-\varepsilon, \gamma+2-2/q-\varepsilon}_p(\B)\oplus \bigoplus_{1\le j\le r} \scrE_{q_j^-}\oplus \underline\scrE_0,
\end{eqnarray}
for all $\varepsilon >0$, where $r\ge0$ satisfies 
$$q_{r+1}^-<\frac{n+1}2-\gamma-2+\frac2q+\varepsilon<q_r^- $$
for all sufficiently small $\varepsilon>0$. 
In particular, 
\begin{eqnarray}
\label{coarse_emb}
(\cH^{s,\gamma}_p(\B), \scrD(\underline\Delta))_{1-1/q,q}\hookrightarrow
\cH_{p}^{s_0,\gamma_0}(\mathbb{B}) \oplus \underline\scrE_0, 
\end{eqnarray}
for 
\begin{eqnarray}
s_0&=&s+2-\frac2q-\varepsilon >\frac{n+1}p \text{ and }\label{s0}\\
\gamma_0&=&\min\Big\{\gamma+2-\frac2q-\varepsilon, \frac{n+1}2-q_1^--\varepsilon\Big\}>\frac{n+1}2.
\label{gamma0}
\end{eqnarray}

From Theorem \cite[Lemma 6.2]{RS2} we recall the following: 

\begin{proposition}\label{SI}
For $s_0>(n+1)/p$ and $\gamma_0>(n+1)/2$ the space $\cH_{p}^{s_0,\gamma_0}(\mathbb{B}) \oplus \underline\scrE_0$ in \eqref{coarse_emb} is a spectrally invariant Banach subalgebra of $C(\B)$ and hence closed under holomorphic functional calculus. 
\end{proposition} 

\begin{theorem} \label{maxreg}
Let $v\in (\cH^{s,\gamma}_p(\B), \scrD(\underline\Delta))_{1-1/q,q}$ be strictly positive. Then there exists a $c>0$ such that $c-mv^{(m-1)/m}\underline \Delta$ is $\mathcal R$-sectorial (see Definition \ref{Rsectorial}) in $\cH^{s,\gamma}_p(\B)$ on the sector $\Lambda$ in \eqref{Lambda}.

According to a theorem of Weis \cite[Theorem 4.2]{W}, the fact that the angle $\theta$ in \eqref{Lambda} can be chosen to be less than $\pi/2$ implies that $c-mu^{m-1}\underline \Delta$ has maximal regularity. In particular this holds for the initial value $u_0$ in \eqref{PME2}.
\end{theorem}

\begin{proof} For strictly positive $v\in (\cH^{s,\gamma}_p(\B), \scrD(\underline\Delta))_{1-1/q,q}\hookrightarrow \cH_{p}^{s_0,\gamma_0}(\B) \oplus \underline\scrE_0$ the spectral invariance implies that 
$$mv^{(m-1)/m} \in \cH_{p}^{s_0,\gamma_0}(\mathbb{B}) \oplus \underline\scrE_0$$
for $s_0$ and $\gamma_0$ defined by \eqref{s0} and \eqref{gamma0}, respectively.
We now infer from \cite[Theorem 6.1]{RS2} that, for suitably large $c>0$, the operator $c-mv^{(m-1)/m} \underline\Delta$ is $\mathcal R$-sectorial of angle $\theta$ for any $\theta \in (0,\pi)$.
Note that, while the situation in \cite{RS2} is different, it is pointed out after Equation (6.3) in \cite{RS2} that the only property needed of $mv^{(m-1)/m}$ is that it belongs to the space $\cH^{s_0,\gamma_0}_p(\B)\oplus\underline\scrE_0$, where the $s$ in \eqref{s0} satisfies $s+2-2/q>|s| + 1+(n+1)/p$ and $\gamma_0>(n+1)/2$ (the reason for this condition is that Lemma \ref{mult}, below, is needed). By the assumptions on $p$, $q$ and $s$ in \eqref{pq} and \eqref{s} this is the case.
\end{proof} 

\subsection{Verifying the assumptions in the Cl\'ement-Li Theorem}\label{H1H2} 
We shall apply Theorem \ref{CL} with $s,\gamma, p$ and $q$ as in Theorem \ref{main}, $X_1=\scrD(\underline\Delta)$ and $X_0= \cH^{s,\gamma}_p(\B)$. 
The interpolation space $(X_0,X_1)_{1-1/q,q}$ is a subset of $C(\B)$ with a stronger topology.
Moreover, $v_0$ is strictly positive by assumption. 
Hence, given any $\rho_0,\rho_1\in\R$ with 
$$0<\rho_0<\frac12\inf\{u_0(z): z\in \B\}\le 2\sup\{ u_0(z):z\in \B\}< \rho_1,$$ 
the set of all functions $v$ in $(X_0,X_1)_{1-1/q,q}$ satisfying $\rho_0<u(z)<\rho_1$ defines a neighborhood $U_{\rho_0,\rho_1}$ of $v_0$ in $(X_0,X_1)_{1-1/q,q}$.
Furthermore, we choose a contour $\Gamma$ in $\{\re z>0\}$ that simply surrounds the interval $[\rho_0,\rho_1]$. With this set-up, we can essentially proceed as in the proof of \cite[Theorem 6.5]{RS2}. For completeness we give the details. 

To verify Condition (H1) for $v_1, v_2$ in $U_{\rho_0,\rho_1}$
we have to estimate the quotient 
$$\|v_1^{(m-1)/m}\underline \Delta - v_2^{(m-1)/m}\underline \Delta\|_{\scrL(X_1,X_0)}/\|v_1-v_2\|_{(X_0,X_1)_{1-1/q,q}}.$$ 
For this it is sufficient to consider 
$$\|v_1^{(m-1)/m}- v_2^{(m-1)/m}\|_{\scrL(X_0)}/\|v_1-v_2\|_{(X_0,X_1)_{1-1/q,q}},$$
where $v_1^{(m-1)/m}$ and $v_2^{(m-1)/m}$ act on $X_0$ as multiplication operators. 

We recall from \eqref{coarse_emb} and Proposition \ref{SI} that $(X_0,X_1)_{1-1/q,q}$ embeds into $\cH^{s_0,\gamma_0}_p(\B)\oplus \underline\scrE_0$ which is spectrally invariant in $C(\B)$.
The identity 
$$(v_1-\lambda)^{-1} - (v_2-\lambda)^{-1} = (v_1-\lambda)^{-1}(v_2-v_1)(v_2-\lambda)^{-1} ,$$
which holds, whenever the inverses exist, implies that, for $v_1, v_2\in U_{\rho_0,\rho_1}$ we can write 
\begin{eqnarray*}
v_1^{(m-1)/m} - v_2^{(m-1)/m} = \frac{v_2-v_1}{2\pi i}\int_{\Gamma} w^{(m-1)/m} (v_1-w)^{-1} (v_2-w)^{-1} \, dw,
\end{eqnarray*}
where equality holds in $\cH^{s_0,\gamma_0}_p(\B)\oplus \underline\scrE_0$. 
In order to estimate the right hand side, we apply \cite[Corollary 3.3]{RS2}, which we restate here for convenience: 
\begin{lemma}\label{mult} 
For $1<p,q<\infty$, $\gamma\in \R$, multiplication defines a bounded map 
$$\cH^{\sigma,(n+1)/2}_q (\B) \times \cH^{s,\gamma}_p(\B)\to \cH^{s,\gamma}_p(\B)$$
provided $\sigma>|s|+(n+1)/q$. 
\end{lemma} 

Recall that $s_0$ in \eqref{coarse_emb} is given by $s_0= s+2-2/q-\varepsilon$. Taking $\varepsilon$ sufficiently small, the conditions in \eqref{pq} and \eqref{s} imply that $s_0=s+2-2/q-\varepsilon > |s|+(n+1)/p$. We can therefore apply Lemma \ref{mult} twice with $\sigma = s_0$ and conclude that, for $v\in X_0=\cH^{s,\gamma}_p(\B)$
\begin{eqnarray*}
\lefteqn{\|(v_1^{(m-1)/m} -v_2^{(m-1)/m})v \|_{X_0}
\le C \|v_1-v_2\|_{\cH^{s_0,\gamma_0}_p(\B)}
}\\
&&
\times \left\|\int_\Gamma w^{(m-1)/m}(v_1-w)^{-1}(v_2-w)^{-1}\, dw\right\|_{\cH^{s_0,\gamma_0}_p(\B)} \|v\|_{X_0}.
\nonumber 
\end{eqnarray*}
As $w$ has positive distance to the range of $v_1$ and $v_2$, respectively, the terms in the integrand are bounded away from zero in $\cH^{s_0,\gamma_0}_p(\B)$, and hence the norm of the integral is bounded. 
Moreover, 
\begin{eqnarray}\nonumber
\lefteqn{\|v_1^{(m-1)/m}\underline \Delta - v_2^{(m-1)/m}\underline \Delta\|_{\scrL(X_1,X_0)}}\\\label{Estum}
&\le& \|v_1^{(m-1)/m} - v_2^{(m-1)/m}\|_{\scrL(X_0)} \le C\|v_1-v_2\|_{(X_0,X_1)_{1-1/q,q}}.
\end{eqnarray}
This shows (H1). 
Condition (H2) is trivially fulfilled in view of the assumption on the forcing term $G$. 
So, the proof of Theorem \ref{main} is complete. \hfill $\Box$

\subsection{Proof of Theorem \ref{SI2}} 
We start with the following observations. 
\begin{enumerate}\renewcommand{\labelenumi}{(\roman{enumi})} 
\item For $s_0>(n+1)/p$, $1<p<\infty$, and $\gamma_0\ge (n+1)/2$, $\cH^{s_0,\gamma_0}_p(\B)$ is a (non-unital) Banach algebra (up to an equivalent norm). This is \cite[Corollary 2.8]{RS1}.

\item For $v_1,v_2\in \underline \scrE_0$, the product $v_1v_2$ is in $\cH^{s_0, \gamma_0}_p(\B)\oplus \underline\scrE_0$ for all $s_0$, $\gamma_0$, $1<p<\infty$. 
Indeed, if $v_j (x,y) = \omega(x)e_j$ with $e_j\in E_0$, $j=1,2$, then 
$$v_1(x,y)v_2(x,y)-\omega(x) e_1(y)e_2(y) = (\omega^2(x)-\omega(x))e_1(y)e_2(y)$$
since the elements of $E_0$ are locally constant on $\partial \B$. The assertion then follows from the fact that $(\omega^2(x)-\omega(x) )e_1(y)e_2(y)$ is smooth and vanishes near $\partial \B$. 

\item For $v_1\in \cH^{\underline s,\underline\gamma}_p(\B)$ and $v_2\in \cH^{\underline s,(n+1)/2}_p(\B)$ we have $v_1v_2$ in $\cH^{\underline s,\underline\gamma}_p(\B)$.
To see this choose a function $\underline x$ on $\B$ coinciding with $x$ near $\partial \B$ and equal to $1$ away from the boundary. Next write 
\begin{eqnarray*}
\cH^{\underline s,\underline \gamma}_p(\B) &=& \underline x^{\underline \gamma-(n+1)/2}\cH^{\underline s,(n+1)/2}_p(\B).
\end{eqnarray*}
Then use (i) and the fact that $\cH^{\underline s,(n+1)/2}_p(\B)$ is an algebra and that multiplication by
 $\underline x^{\underline \gamma-(n+1)/2}$ maps $\cH^{\underline s,(n+1)/2}_p(\B)$ to 
 $\cH^{\underline s, \underline \gamma}_p(\B)$.

\item Let $N\in \N$ be so large that $Nq_1^-<(n+1)/2-\underline\gamma$ and let 
$v_1, \ldots, v_N\in \scrE_{q_1^-}\oplus \ldots \oplus \scrE_{q_r^-}$.
Then $ v_1\cdot\ldots\cdot v_N \in \cH^{\underline s, \underline \gamma}_p(\B).$ 
Namely, the product is a finite sum of terms of the form 
$x^{-q} \ln^Kx\; \omega^N(x)c(y)$ with $q\le Nq_1^-$, $K\le N$ and $c\in C^\infty(\partial \B)$. 
It therefore belongs to $\cH^{\underline s, \underline \gamma}_p(\B)$. 
\item For $s_0>(n+1)/p$, $1<p<\infty$, and $\gamma_0> (n+1)/2$, $\cH^{s_0,\gamma_0}_p(\B)\oplus \underline \scrE_0$ is a spectrally invariant Banach subalgebra of $C(\B)$ by Proposition \ref{SI}. 
\end{enumerate} 
In order to see the spectral invariance, we consider an element $w\in \B$ that is invertible in $C(\B)$; i.e., $w$ is bounded away from zero on $\B$. Since the elements of $\cH^{\underline s, \underline \gamma}_p(\B)$ and those of $\scrE_{q_1^-}\oplus \ldots \oplus \scrE_{q-r^-}$ all vanish at $x=0$, we infer from (i), (ii) and (iii) that we can write $w$ in the form 
$$w=e_0+ V_0 + u_0,$$
where $e_0\in \underline \scrE_0$ is bounded away from zero, $V_0$ is a sum of terms of the form 
$v_1\cdot \ldots \cdot v_m$ for some $m$ with at least one factor in $\scrE_{q_1^-}\oplus \ldots \oplus\scrE_{q_r^-}$ and $u_0\in \cH^{\underline s, \underline\gamma}_p(\B)$. 

Choose $u_1 \in \cH^{\underline s,\underline\gamma}_p(\B)$ such that $e_0 + u_1$ is invertible in $C(\B)$. Then we can write 
\begin{eqnarray*}
w&=& (e_0+u_1)+V_0 + (u_0-u_1)\\ &=& (e_0+u_1)\left(1+ (e_0+u_1)^{-1} V_0 + (e_0+u_1)^{-1} (u_0-u_1) \right) .
\end{eqnarray*}
According to (v), $(e_0+u_1)^{-1} = \tilde e_0 + \tilde u_1$ with $\tilde e_0 \in \underline \scrE_0$ and $\tilde u_1 \in \cH^{\underline s, \underline\gamma}_p(\B)$. Hence we can represent $w$ in the form 
$$w= (e_0+u_1)( 1+V+u) ,$$
where $V=\tilde e_0V_0$ has the same structure as $V$, and $u = (\tilde e_0 + \tilde u_1)(u_0-u_1)+ \tilde uV_0\in \cH^{\underline s,\underline\gamma }_p(\B)$ by (iii). 

Next consider the finite geometric series expansion
\begin{eqnarray}
\label{geomseries}
(1+y)^{-1} = 1-y+y^2\ldots +(-1)^Ny^N(1+y)^{-1}
\end{eqnarray}
for $y= V+u$. As $V$ contains at least one factor in $\scrE_{q_1^-}\oplus \ldots \oplus\scrE_{q_r^-}$, it belongs to some space $\cH^{\underline s,\gamma_0}_p(\B)$ with $\underline\gamma\ge\gamma_0>(n+1)/2$ so that $1+V+u\in \cH^{\underline s,\gamma_0}_p(\B)\oplus \underline \scrE_0$. By (v), $(1+V+u)^{-1}\in \cH^{\underline s,\gamma_0}_p(\B)\oplus \underline \scrE_0$.
Due to our assumption on $V$ and (iv), $y^N = (V+u)^N\in\cH^{\underline s,\underline\gamma}_p(\B)$. We then conclude from \eqref{geomseries} and (iii) that $(1+V+u)^{-1}\in \scrB$. Since also $(e_0+u_1)^{-1}$ belongs to $\scrB$ so does $w^{-1}$. 

In order to see that $\scrB$ is closed under holomorphic functional calculus, let $v\in \scrB$ and let $f$ be a function that is holomorphic in an open neighborhood $U$ of the range of $v$. For a piecewise differentiable contour $\scrC$ in $U$ that simply surrounds the range of $v$, 
$$ f(v) = \frac1{2\pi i} \int_{\scrC} f(z) (v-z)^{-1} \, dz.$$
We claim that this is an element of $\scrB$. Indeed, $z\mapsto v-z$ is a continuous map with values in $\scrB$, because we can write $z= \omega(x) z + (1-\omega(x))z \in \underline \scrE_0\oplus \cH^{\underline s,\underline \gamma}_p(\B)$. Since inversion is continuous in any Banach algebra, the integrand is a continuous $\scrB$-valued function, so that the integral furnishes an element of $\scrB$, and Theorem \ref{SI2} is proven. \hfill $\Box$

\section{Appendix} 
\subsection{The spaces $\cH^{s,\gamma}_p(\B)$}
\label{HsgKsg}
For $\gamma\in \R$ define the map 
$$\cS_\gamma: C^\infty_c(\R^{1+n}_+)\to C^\infty_c(\R^{1+n}), \quad v(x,y)\mapsto e^{(\gamma- \frac{n+1}2)x}v(e^{-x},y).$$
\begin{definition} Let $s,\gamma\in \R$, $1<p<\infty$. 
Given coordinate charts $\kappa_j:U_j\subseteq\partial\B \to \R^n$, $j=1,\ldots ,N$, for a neighborhood of $\partial \B$ and a subordinate partition of unity $\{\phi_j: j=1,\ldots ,N\}$, 
\begin{eqnarray}\label{hsg1}\phantom{xxxx}
\cH^{s,\gamma}_p(\B)= \{ u\in H^s_{p,loc}(\Int(\B)):\cS_\gamma (1\otimes \kappa_j)_*(\phi_j u)\in H^s_p(\R^{1+n}), j=1,\ldots,N\}. 
\end{eqnarray}
\end{definition}

$\cH^{s,\gamma}_p(\B)$ is a Banach space independent of the choices made. 

\subsection{The closed extensions of $\Delta$ and $\widehat \Delta$ }
In this subsection we will recall (and slightly extend) some of the results on the structure of the domains of the closed extensions of the Laplacian, adapted from Sections 3 and 6 in \cite{SS2}, starting from equation \eqref{LB} 
$$\Delta = x^{-2} \left((-x\partial_x)^2-(n-1+H(x))(-x\partial_x) + \Delta_{h(x)}\right).$$

The Mellin transform $\scrM v$ of a function $v\in C^\infty_c(\R_+)$ is given by 
$$\scrM v(z) = \int_0^\infty x^{z-1} v(x) dx.$$
In view of the fact that $\scrM ((-x\partial_x)v) (z) = z\scrM v(z)$, we can write for $u \in C^\infty_c(\R_+\times\partial \B)$ 
$$(\Delta u) (x,y) = x^{-2}\scrM^{-1}_{z\to x} \left(z^2 -(n-1+H(x))z +\Delta_{h(x)}\right) \scrM_{x\to z}u(x,y).$$ 
We define two polynomials in $z$, namely
\begin{eqnarray}
f_0(z) &=& z^2-(n-1)z-\Delta_{h(0)}\\
f_1(z) &=& -(\partial_x H)_{|x=0}\, z+\partial_x(\Delta_{h(x)})_{|x=0}=: -H'z+\Delta'. \label{f1}
\end{eqnarray}
They are the first Taylor coefficients in the expansion of 
$$z^2 -(n-1+H(x))z +\Delta_{h(x)}$$ 
with respect to $x$ (recall that $H_{|x=0}=0$) and take values in differential operators on the cross-section $\partial \B$. In fact, $f_0(z)=\sigma_M(\Delta) (z)$, is the principal Mellin symbol, see \eqref{conormal}. On the right hand side of \eqref{f1}, $(\partial_x H)_{|x=0}$ is a function of $y\in \partial \B$, and $\partial_x(\Delta_{h(x)})_{|x=0}$ is a second order differential operator without zero order term. 

Next introduce the meromorphic functions 
\begin{eqnarray}
g_0(z) &=& 1\\
g_1(z) &=& -(f_0(z-1))^{-1} f_1(z) \label{g1}. 
\end{eqnarray}
The background is that then the Mellin product formula implies that 
$$\sum_{j=0}^m f_{m-j}(z-j)g_j(z)f_0(z)^{-1} =
\begin{cases} 1\,; &m=0\\
0\,;& m=1. 
\end{cases} 
$$

\begin{theorem} Let $s,\gamma\in \R$ and $1<p<\infty$. 
Then there exist subspaces $\scrE$ and $\widehat\scrE$ of $C^\infty(\R_+\times \partial \B)$ of the same finite dimension such that, for every cut-off function $\omega$, 
\begin{eqnarray*}
\scrD(\Delta_{\max})=\scrD(\Delta_{\min}) \oplus \omega\scrE\text{ and }
\scrD(\widehat\Delta_{\max})= \scrD(\widehat\Delta_{\min})\oplus \omega \widehat\scrE.
\end{eqnarray*}
If, in addition, $\sigma_M(\Delta)(z)$ is invertible as a second order pseudodifferential operator $($or, equivalently, as a bounded operator $H^2(\partial \B)\to L^2(\partial \B))$ for every $z\in \C$ with $\re(z)= \frac{n+1}2-\gamma-2$, then 
\begin{eqnarray*}
\scrD(\Delta_{\min}) =\cH^{s,\gamma}_p(\B) \text{ and } 
\scrD(\widehat\Delta_{\min})= \cK^{s,\gamma}_p(\R_+\times\partial\B) .
\end{eqnarray*}
The spaces $\scrE$ and $\widehat\scrE$ are independent of $s$ and $p$. 
\end{theorem} 

The following result describes the space $\widehat{\scrE}$ associated with the maximal extension of the model cone operator $\widehat \Delta$. 

\begin{theorem}\label{max_domain1}
Let $\sigma\in I_\gamma$, see \eqref{Igamma}, be a pole of $\sigma_M(\Delta)^{-1}$. Define 
$$G_\sigma^{(0)}:C^\infty_c(\R_+\times \partial \B)\to C^\infty(\R_+\times \partial \B)\cong C^\infty_c(\R_+,C^\infty(\partial \B))$$ 
by 
 \begin{eqnarray}
\label{G0}
(G_\sigma^{(0)} u)(x)=
 (2\pi i)^{-1} \int_{|z-\sigma|=\varepsilon} x^{-z}f_0^{-1}(z)\scrM{u}(z)\,d z,
\end{eqnarray}
where $\varepsilon>0$ is chosen sufficently small.
Then 
 $$\widehat\scrE=\mathop{\mbox{\Large$\oplus$}}_{\sigma\in I_\gamma} \widehat\scrE_\sigma,\qquad 
 \widehat\scrE_\sigma=\mathrm{range}\,G_\sigma^{(0)}.$$
\end{theorem}

We next describe the space $\scrE$ for the maximal extension of $\Delta$ in Theorem \ref{Dom}. 

\begin{theorem}\label{theta}
Let $\sigma\in I_\gamma$ be as in Theorem \ref{max_domain1}.
Define $G_\sigma^{(0)} $ as above. In case $\sigma-1 \ge \frac{n+1}2-\gamma-2$ introduce additionally 
$G_\sigma^{(1)}:C^\infty_c(\R_+\times \partial \B)\to C^{\infty}(\R_+\times \partial \B)$ by 
\begin{equation}\label{eq:sigma0}
 (G_\sigma^{(1)} u)(x)=\frac{x}{2\pi i}\,\int_{|z-\sigma|=\varepsilon} x^{-z}g_1(z)\,\Pi_\sigma(f_0^{-1}\,\scrM{u})(z)\,d z,
\end{equation}
where $\Pi_\sigma$ is the projection onto the principal part of the Laurent series. Let
\begin{equation}\label{eq:gsigma}
 G_\sigma:=
 \begin{cases} 
G_\sigma^{(0)}; & \sigma-1 <\frac{n+1}2-\gamma-2\\
G_\sigma^{(0)}+G_\sigma^{(1)}; & \sigma-1 \ge\frac{n+1}2-\gamma-2.
 \end{cases} 
\end{equation}
Then 
 $$\scrE=\mathop{\mbox{\Large$\oplus$}}_{\sigma\in I_\gamma} \scrE_\sigma,
 \qquad \scrE_\sigma=\mathrm{range}\,G_\sigma.$$
Moreover, the following map is well-defined and an isomorphism$:$
\begin{equation}\label{eq:isom1}
 \Theta_\sigma:\scrE_\sigma\longrightarrow\widehat{\scrE}_\sigma,\quad 
 G_\sigma(u)\mapsto G^{(0)}_\sigma(u).
\end{equation}
Consequently, we obtain an isomorphism 
\begin{eqnarray}
\label{isotheta}
\Theta=\oplus_{\sigma\in I_\gamma} \Theta_\sigma: \scrE\to \widehat\scrE. 
\end{eqnarray}
\end{theorem}

The reason for distinguishing the cases in \eqref{eq:gsigma} is that, for $ \sigma-1 <\frac{n+1}2-\gamma-2$, the range of $\omega G_\sigma^{(1)}$ is already contained in $\cH^{s+2,\gamma+2}_p(\B)$.
In order to avoid distinguishing too many cases, we will include the range of $G^{(1)}_\sigma$ in the sequel. 

\subsection{The Spaces $\scrE_{q_j^-}$}

Recall from \eqref{inverse} that 
\begin{eqnarray}\label{f0inverse}
f_0(z)^{-1} = \sigma_M(\Delta)^{-1}(z) = \sum_{j=0}^\infty \frac{\pi_j}{(z-q_j^+)(z-q_j^-)},
\end{eqnarray}
where $\pi_j$ is the orthogonal projection in $L^2(\partial \B)$ onto the eigenspace $E_j$ of the eigenvalue $\lambda_j$ of $\Delta_{h(0)}$. 

\subsubsection*{The spaces $\scrE_{q_j^-}$, $j\ge1$}
By \eqref{f0inverse}, $q_j^-$ is a simple pole of $f_0^{-1}$ with residue 
$$(q_j^--q_j^+)^{-1} \pi_j.
$$ 
Since $\scrM u$ is holomorphic on $\C$, \eqref{G0} and the residue theorem implies that 
\begin{eqnarray*}
(G_{q_j^-}^{(0)}u)(x) &=& (2\pi i(q_j^--q_j^+) )^{-1} \int_{|z-q_j^-|=\varepsilon}\frac{x^{-z} }{z-q_j^-}\pi_j (\scrM u(z)) \, dz\\
&=& (q_j^--q_j^+)^{-1}x^{-q_j^-}\pi_j (\scrM u(q_j^-)) .
\end{eqnarray*}
We conclude that the range of $G_{q_j^-}^{(0)}$ is the finite-dimensional space of all functions $v$ of the form $v(x,y) = x^{-q_j^-}e(y)$ with $e\in E_j$. 

Near $z=q_j^-$ we can write 
\begin{eqnarray}
g_1(z) \equiv \beta_{q_j^-}^{(-1)}(z-q_j^-)^{-1} +\beta^{(0)}_{q_j^-}
\end{eqnarray}
modulo terms that are holomorphic near $q_j^-$ and vanish in $q_j^-$. 
Clearly, $\beta_{q_j^-}^{(-1)} =0$ unless $f_0^{-1} $ has a pole in $q_j^-$. We find
\begin{eqnarray*}
(G^{(1)}_{q_j^-} u)(x) &=& \frac{x}{2\pi i(q_j^--q_j^+)} 
\int_{|z-q_j^-|=\varepsilon}x^{-z} (\beta_{q_j^-}^{(-1)}(z-q_j^-)^{-1} +\beta^{(0)}_{q_j^-}) 
\pi_j \Big(\frac{\scrM u(q_j^-)}{z-q_j^-}\Big)\, dz\\
&=& \frac{x^{1-q_j^-}}{2\pi i(q_j^--q_j^+)} \Big(\beta_{q_j^-}^{(0)} + \ln x\beta_{q_j^-}^{(-1)}\Big)\pi_j \scrM u(q_j^-).
\end{eqnarray*}

We thus obtain:

\begin{lemma}\label{Esigma}
Let $\omega$ be a cut-off function near $\partial \B$ and $j\ge1$. Then
\begin{align*}
\widehat \scrE_{q_j^-}
&= \{ u\in C^\infty(\R_+\times \partial \B): u(x,y) = \omega(x) x^{-q_j^-} e(y); e \in E_j \} \text{ and} \\
\scrE_{q_j^-}&= \{ u\in C^\infty(\Int \B): u(x,y) = \omega(x) x^{-q_j^-} e(y)+
x^{-q_j^-+1}(\beta_{q_j^-}^{(0)} + \ln x \beta_{q_j^-}^{(-1)}) e(y) ; e \in E_j \}.
\end{align*} 
For $\scrE_{q_j^-}$ we identify here a neighborhood of $\partial \B$ with the collar $\R_+\times \partial \B$. 
\end{lemma} 

\subsubsection*{The spaces $\scrE_0$ and $\widehat \scrE_0$ for $n\ge 2$}
For $n\ge 2$, the pole in $0=q_0^-$ is simple. As $q_0^+=n-1$, the residue of $f_0^{-1}$ in $z=0$ is $-(n-1)^{-1}\pi_0$, and the residue theorem implies that 
 \begin{eqnarray*}
(G^{(0)}_0u)(x) &=&-\frac{x^0}{2\pi i(n-1)} \int_{|z|=\varepsilon} \frac{x^{-z}}z \pi_0(\scrM u(z)) \, dz= -\frac1{n-1} \pi_0\scrM u(0). 
\end{eqnarray*}

For $G^{(1)}_0$ we have the expression
\begin{eqnarray*}
(G^{(1)}_0u)(x) &=&-\frac x{2\pi i} \int_{|z|=\varepsilon} x^{-z} (f_0(z-1))^{-1} (-H' z+\Delta')
\Pi_0((f_0(z))^{-1} \scrM u(z)) \, dz.
\end{eqnarray*}
Equation \eqref{f0inverse} implies that the principal part of the Laurent expansion is given by 
$$-\frac1{n-1}\Pi_0\frac{\pi_0 \scrM u(z)}{z} =-\frac1{n-1} \frac{\pi_0 \scrM u(0)}{z}.$$
Moreover, we observed that $\Delta'$ has no zero order term. Since $\pi_0$ projects onto the constant functions, $\Delta'\pi_0=0$. We obtain
\begin{eqnarray*}
(G^{(1)}_0u)(x) &=&-\frac x{2\pi i(n-1) } \int_{|z|=\varepsilon} x^{-z} (f_0(z-1))^{-1} H' 
\pi_0 \scrM u(0)) \, dz.
\end{eqnarray*}
Hence there will be no contribution from $G_0^{(1)}$, unless $(f_0(z-1))^{-1}$ has a pole in $z=0$ or, equivalently, if $f_0^{-1}$ has a pole in $-1$. So let us assume that this is the case. 
Since \eqref{gamma_new} implies that $\frac{n+1}2-\gamma-2$ is not a pole of $f_0^{-1}$, $-1$ necessarily is one of the elements in the set $\{q_1^-,\ldots, q_k^-\}$, say $-1=q_\ell^-$. Since $q_\ell^+= n$, the residue in $z=0$ is 
\begin{eqnarray*}
-\frac{\pi_\ell}{1+q_\ell^+}= -\frac{\pi_\ell}{n+1}. 
\end{eqnarray*}
Thus 
\begin{eqnarray*}
(G^{(1)}_0u)(x) &=&-\frac x{n^2-1} \pi_\ell (H' \pi_0\scrM u(0)).
\end{eqnarray*}
We conclude: 

\begin{lemma}\label{nge2pole} Let $n\ge2$,
and let $\omega$ be a cut-off function near $\partial \B$. 
Then 
\begin{eqnarray*}
\widehat \scrE_0&=& \{ u\in C^\infty(\R_+\times \partial \B): u(x,y) = \omega(x) e(y); e\in E_0\} .
\end{eqnarray*}
If $-1$ is a pole of $f_0^{-1}=\sigma_M(\Delta)^{-1}$, say, $-1=q_\ell^-$, then 
\begin{eqnarray*}
\scrE_0 = \{u\in C^\infty(\Int(\B)): u(x,y) = \omega(x)(e_0+ x\pi_\ell((\partial_{x}H)_{|x=0}e_0)); e_0\in E_0\}; 
\end{eqnarray*}
Otherwise $\scrE_0 =\widehat \scrE_0$ with the identification of a neighborhood of $\partial \B$ with $\R_+\times \partial \B$. 
\end{lemma}

\begin{remark}\label{omit1}
If $-1=q_\ell^-$ is a pole of $f_0^{-1}$, the functions $u(x,y) = \omega(x) x\pi_\ell (\partial_x H|_{x=0})e_0$, $e_0\in E_0$, form a subset of $\scrE_{q^-_\ell}= \scrE_{-1}$. In the definition of $\scrD(\underline\Delta)$ in \eqref{domD} we can therefore replace $\scrE_0$ as defined in Lemma \ref{nge2pole} by $\{u\in C^\infty(\Int(\B)): u(x,y) = \omega(x) e(y); e\in E_0\}$.
\end{remark}

\subsubsection*{The spaces $\scrE_0$ and $\widehat \scrE_0$ for $n=1$. }
For $n=1$, we have $q_0^-=q_0^+=0$ and therefore a double pole of $(f_0(z))^{-1} = 
(\sigma_M(\Delta)(z))^{-1}$ in $z=0$, and \eqref{f0inverse} implies that near $z=0$
$$f_0(z) \equiv \frac{\pi_0}{z^2}$$
modulo terms holomorphic in $z=0$. Writing 
$$(\scrM u)(z) \equiv (\scrM u)(0) + (\scrM u)'(0)z + z^2g(z)$$
for a holomorphic function $g$ near $z=0$, we see that 
\begin{eqnarray*}
\Pi_0(f_0^{-1} \scrM u)(z) &=&\Pi_0\left(\frac{\pi_0}{z^2}((\scrM u)(0) + (\scrM u)'(0) z + z^2g(z)) \right) \\ &=& \frac{\pi_0((\scrM u)(0))}{z^2} + \frac{\pi_0((\scrM u)'(0))}{z} .
\end{eqnarray*}
The residue theorem implies that, for $u\in C^\infty_c(\R_+, C^\infty(\partial \B))$, 
\begin{eqnarray}
(G_0^{(0)}u)(x) &=& \frac1{2\pi i} \int_{|z|=\varepsilon} x^{-z} \left(\frac{\pi_0((\scrM u)(0))}{z^2} + \frac{\pi_0((\scrM u)'(0))}{z} \right) \, dz\nonumber \\
&=& -x^0\ln x \, \pi_0((\scrM u)(0)) + x^0 \pi_0 ((\scrM u)'(0)).
\label{G00}
\end{eqnarray}

Moreover, 
\begin{eqnarray}\lefteqn{(G_0^{(1)}u)(x)}\nonumber\\ 
&=& \frac{x}{2\pi i} \int_{|z|=\varepsilon} x^{-z} g_1(z) \Pi_0(f_0^{-1} \scrM u)(z) \, dz 
\nonumber\\
&=&- \frac{x}{2\pi i} \int_{|z|=\varepsilon} x^{-z} (f_0(z-1))^{-1} (-H'z+\Delta')\left(\frac{\pi_0((\scrM u)(0))}{z^2} + \frac{\pi_0((\scrM u)'(0))}{z} \right) \, dz 
\nonumber\\
&=& \frac{x}{2\pi i} \int_{|z|=\varepsilon} x^{-z} (f_0(z-1))^{-1}\left(H' \left(\frac{\pi_0((\scrM u)(0))}{z} + \pi_0((\scrM u)'(0)) \right)\right) \, dz ,
\nonumber
\end{eqnarray}
 where we have used the fact that $\Delta'$ has no zero order term and thus vanishes on the range of $\pi_0$. 

1. In case $z=0$ is {\em not} a pole, we conclude that 

\begin{eqnarray*}
(G_0^{(1)}u)(x) = x (f_0(-1))^{-1} (H'\pi_0((\scrM u)(0))).
\end{eqnarray*}
Here, $(f_0(-1))^{-1} = (1+\Delta_{h(0)})^{-1} $. 

2. In case $z=0$ {\em is} a pole of $(f_0(z-1))^{-1}$, i.e. $-1$ is a pole of $f_0^{-1}$, the considerations made before Lemma \ref{nge2pole} show that there must be an $\ell\in \{1,\ldots, k\} $ with $q_\ell^-=-1$. Then $q_\ell^+=1$ and, near $z=0$, 
\begin{eqnarray*}
(f_0(z-1))^{-1} \equiv -\frac12 \frac{\pi_\ell}{z} +S_0 
\end{eqnarray*}
modulo holomorphic functions that vanish to first order in $z=0$. 
Hence 
\begin{eqnarray*}\lefteqn{(f_0(z-1))^{-1}\left(H' \left(\frac{\pi_0((\scrM u)(0))}{z} + \pi_0((\scrM u)'(0)) \right)\right)}\\
&=& 
-\frac12\pi_\ell \left(H'\left(\frac{\pi_0((\scrM u)(0))}{z^2} + \frac{\pi_0((\scrM u)'(0))}{z} \right)\right)
+S_0 \left(H'\frac{\pi_0((\scrM u)(0))}{z}\right)
\end{eqnarray*}
modulo functions that are holomorphic near $z=0$. 

Inserting this into the formula for $G^{(1)}_0$, we find that 
\begin{eqnarray*}
\lefteqn{(G_0^{(1)}u)(x)} \\
&=& \frac12 x\ln x \,\pi_\ell \left(H'\pi_0(\scrM u(0))\right)
-\frac12 x \,\pi_\ell\left(H'\pi_0((\scrM u)'(0) \right) -\frac12 x\, S_0(H'\pi_0(\scrM u(0))). 
\end{eqnarray*}

As a consequence, we obtain

\begin{lemma}\label{n1gammapos}
Let $n=1$ and $\omega$ a cut-off function near $\partial \B$. Then
$$ \widehat\scrE_0 = \{u\in C^\infty(\R_+,C^\infty(\partial \B)): u(x,y) = \omega(x)(e_0+e_1\ln x); \ e_0,e_1\in E_0\} .$$
If $-1$ is not a pole of $f_0^{-1} = \sigma_M(\Delta)^{-1}$, then 
\begin{eqnarray*}
\lefteqn{\scrE_0 = \Big\{u\in C^\infty(\Int(\B)): }\\
&&\left.u(x,y) = \omega(x)\left(e_0(y)+\ln x\, e_1 +x(1+\Delta_{h(0)})^{-1} (\partial_xH|_{x=0} e_1)(y)\right); \, e_0,e_1\in E_0\right\}.
\end{eqnarray*}
Otherwise
\begin{eqnarray*}
\lefteqn{\scrE_0 = \Big\{u\in C^\infty(\Int(\B)):u(x,y) =\omega(x) \Big(e_0(y) -\frac x2\pi_\ell(\partial_xH|_{x=0}) e_0)(y)}\\ 
&&+\ln x e_1(y)+\frac x2\left( \ln x\, \pi_\ell (\partial_xH|_{x=0} e_1)(y) +S_0(\partial_xH|_{x=0} e_1)(y)\right)\Big) ; \, e_0,e_1\in E_0\Big\}.
\end{eqnarray*}
\end{lemma} 

\begin{remark}\label{omit2} 
Similarly to Remark \ref{omit1}, we can omit the terms $\omega(x) x\pi_\ell(\partial_xH|_{x=0} )e_0$, $e_0\in E_0$, appearing in Lemma \ref{n1gammapos} in the definition of the domain of $\underline \Delta$ , since they are contained already in the space $\scrE_{q_\ell^-}$. 
\end{remark}

\subsection{$H^\infty$-calculus, $\mathcal R$-boundedness, and maximal regularity}
For completeness of the exposition we recall a few basic definitions. A good reference is \cite{DHP}. 
Let $X_0$ and $X_1$ be Banach spaces with $X_1$ densely and continuously embedded in $X_0$. Moreover, let $-B\in \scrL(X_1,X_0)$ be the infinitesimal generator of an analytic semigroup with domain $\scrD(B)=X_1$ and $1<q<\infty$, $T>0$. In $L^q(0,T;X_0)$ consider the initial value problem
\begin{eqnarray}\label{PMR}
\partial_t u + Bu =f, \quad u(0) = u_0
\end{eqnarray}
for data $f\in L^p(0,T; X_0)$ and $u_0\in (X_0,X_1)_{1-1/q,q}$. 

\begin{definition}\label{MaxReg}$B$ has maximal $L^q$-regularity, if, 
for every $u_0\in (X_0,X_1)_{1-1/q,q}$ and every $f\in L^q(0,T; X_0)$,
the initial value problem \eqref{PMR} has a unique solution $u\in W^{1,q}(0,T; X_0)\cap L^q(0,T;X_1)$ that depends continuously on $u_0$ and $f$. 
\end{definition} 

The $H^\infty$-calculus for sectorial operators was introduced by A. McIntosh. 
Let $\Lambda_\theta$ be as in \eqref{Lambda}, let $X_0$ and $X_1$ be as above and let $B$ be a closed linear operator with domain $\scrD(B)=X_1$. Suppose that there exists a $C> 0$ such that
$$\|\lambda(\lambda-B)^{-1} \|_{\scrL(X_0)}\le C $$
for all $\lambda \in \Lambda_\theta$. Then one can define
$$ f(B)=\frac1{2\pi i}\int_{\partial \Lambda_\theta} f(\lambda)(\lambda-B)^{-1}\, d\lambda$$ 
for $f\in H^\infty_0(\Lambda_\theta)$, the space of bounded holomorphic functions on $\C\setminus \Lambda$ with additional decay properties near zero and infinity so that the integral makes sense.

\begin{definition} The operator $B$ is said to have a bounded $H^\infty$-calculus with respect to $\Lambda_\theta$, if there exists a constant $C$ such that, for all $f$ in $H^\infty_0(\Lambda_\theta)$, 
$$\|f(B)\|_{\scrL(E_0)} \le C\|f\|_\infty.$$ 
\end{definition} 

\begin{definition}\label{Rsectorial}
We call $B$ {\em $\mathcal R$-sectorial of angle $\theta$}, if there exists a constant $C>0$ such that for any choice of $\lambda_{1},...,\lambda_{N}\in \C\setminus \Lambda_{\theta}$, $x_{1},...,x_{N}\in X_0$, $N\in\mathbb{N}$, and the sequence $\{\epsilon_{\rho}\}_{\rho=1}^{\infty}$ of the Rademacher functions,
we have 
\begin{eqnarray}\label{RS1}
\Big\|\sum_{\rho=1}^{N}\epsilon_{\rho}\lambda_{\rho}(\lambda_{\rho}-B)^{-1}x_{\rho}\Big\|_{L^{2}(0,1;X_0)} \leq C \Big\|\sum_{\rho=1}^{N}\epsilon_{\rho}x_{\rho}\Big\|_{L^{2}(0,1;X_0)}.
\end{eqnarray}
\end{definition} 

We recall the following facts, which hold in UMD Banach spaces: 
\begin{proposition}\label{prop:MR} 
{\rm (a)} The existence of a bounded $H^\infty$-calculus implies the $\mathcal R$-sectoriality for the same sector according to Cl\'ement and Pr\"uss, \cite[Theorem 4]{CP}.

{\rm (b)} Every operator that is $\mathcal R$-sectorial on $\Lambda_\theta$ for some $\theta<\pi/2$, has maximal $L^q$-regularity, $1<q<\infty$, 
see Weis \cite[Theorem 4.2]{W}.
\end{proposition}

All Mellin-Sobolev spaces $\cH^{s,\gamma}_p(\B)$ and $\cK^{s,\gamma}_p(\R_+\times \partial \B)$ used here are UMD Banach spaces, hence the existence of a bounded $H^\infty$-calculus on $\Lambda_\theta$ for $\theta<\pi/2$ implies maximal $L^q$-regularity.


\begin{thebibliography}{99}
\bibitem{Amann} H. Amann. {\em Linear and quasilinear parabolic problems, Vol. I Abstract linear theory}. Monographs in Mathematics {\bf89}, Birkh\"auser Verlag 1995.

\bibitem{BV} E. Bahuaud, B. Vertman. {\em Long-time existence of the edge Yamabe flow.} J. Math. Soc. Japan {\bf 71}, no. 2, 651--688 (2019).

\bibitem{BBGM24} E. Berchio, M. Bonforte, G. Grillo, M. Muratori. {\em The fractional porous medium equation on noncompact Riemannian manifolds.} Math. Ann. {\bf 389}, no. 4, 3603--3651 (2024). 

\bibitem{BS} G. Bourdaud, W. Sickel. {\em Composition operators on function spaces with fractional order of smoothness}. RIMS Kokyuroku Bessatsu {\bf B26}, 93--132 (2011).

\bibitem{BS88} J. Br\"uning, R. Seeley. {\em An index theorem for first order regular singular operators}. Amer. J. Math. {\bf 110}, no. 4, 659--714 (1988).

\bibitem{CL} P. Cl\'ement, S. Li. {\em Abstract parabolic quasilinear equations and application to a groundwater flow problem.} Adv. Math. Sci. Appl. {\bf 3} (Special Issue), 17--32 (1993/94).

\bibitem{CP} P. Cl\'{e}ment, J. Pr\"uss. {\em An operator-valued transference principle and maximal regularity on vector-valued $L_p$-spaces}. In: G. Lumer and L. Weis (eds.), Proc. of the 6th. International Conference on Evolution Equations. Marcel Dekker (2001).

\bibitem{CSS1} S. Coriasco, E. Schrohe, J. Seiler. {\em Differential operators on conic manifolds: Maximal regularity and parabolic equations}. Bull. Soc. Roy. Sci. Li\`ege {\bf 70}, 207--229 (2001).

\bibitem{DHP} R. Denk, M. Hieber, J. Pr\"uss. {\em $R$-boundedness, Fourier multipliers and problems of elliptic and parabolic type}. Mem. Amer. Math. Soc. {\bf 166} (2003).

\bibitem{GKM2} J. Gil, T. Krainer, G. Mendoza. {\em Geometry and spectra of closed extensions of elliptic cone operators.} Canad. J. Math. {\bf59}, no. 4, 742--794 (2007).

\bibitem{GKM1} J. Gil, T. Krainer, G. Mendoza. {\em Resolvents of elliptic cone operators.} J. Funct. Anal. {\bf 241}, no. 1, 1--55 (2006).

\bibitem{GM1} J. Gil, G. Mendoza. {\em Adjoints of elliptic cone operators.} Amer. J. Math. {\bf125}, no.2, 357--408 (2003). 

\bibitem{GHUV} D. Grieser, S. Held, H. Uecker, B. Vertman. {\em Phase transitions and minimal interfaces on manifolds with conical singularities}. {\tt arXiv:2403.07178}. 

\bibitem{Lesch} M. Lesch. {\em Operators of Fuchs type, conical singularities, and asymptotic methods}. Teubner-Texte Math. {\bf 136}, Teubner-Verlag, 1997.

\bibitem{LR2} P. T. P. Lopes, N. Roidos. {\em Existence of global attractors and convergence of solutions for the Cahn-Hilliard equation on manifolds with conical singularities}. J. Math. Anal. Appl. {\bf 531}, no. 2, 127851 (2024). 

\bibitem{LR1} P. T. P. Lopes, N. Roidos. {\em Smoothness and long time existence for solutions of the Cahn-Hilliard equation on manifolds with conical singularities}. Monatshefte f\"ur Mathematik {\bf 197}, 677--716 (2022).

\bibitem{Lunardi} A. Lunardi. {\em Interpolation theory}. Lecture Notes Scuola Normale Superiore {\bf 16}, Edizioni della Normale (2018).

\bibitem{MP73} V. G. Maz'ya, V. G., B.A. Plamenevsk\u ii. {\em The behavior of the solutions of quasilinear elliptic boundary value problems in the neighborhood of a conical point}. (Russian) Zap. Nauchn. Sem. Leningrad. Otdel. Mat. Inst. Steklov. (LOMI) {\bf38}, 94--97 (1973).

\bibitem{MRS15} R. Mazzeo, Y. Rubinstein, N. Sesum. {\em Ricci flow on surfaces with conic singularities}. Anal. PDE {\bf 8}, no. 4, 839--882 (2015).

\bibitem{Ro1} N. Roidos. {\em Complex powers for cone differential operators and the heat equation on manifolds with conical singularities}. Proceedings of the Amer. Math. Soc. {\bf 146}, no. 7, 2995--3007 (2018).

\bibitem{RoSa} N. Roidos, A. Savas-Halilaj. {\em Curve shortening flow on Riemann surfaces with conical singularities}. Math. Ann. {\bf 390}, 2337--2411 (2024).

\bibitem{RS1} N. Roidos, E. Schrohe. {\em Bounded imaginary powers of cone differential operators on higher order Mellin-Sobolev spaces and applications to the Cahn-Hilliard equation}. J. Differential Equations {\bf 257}, no. 3, 611--637 (2014). 

\bibitem{RS2} N. Roidos, E. Schrohe. {\em Existence and maximal $L^{p}$-regularity of solutions for the porous medium equation on manifolds with conical singularities}. Comm. Partial Differential Equations {\bf41}, no. 9, 1441--1471 (2016).

\bibitem{RS3} N. Roidos, E. Schrohe. {\em Smoothness and long time existence for solutions of the porous medium equation on manifolds with conical singularities}. Comm. Partial Differential Equations {\bf43}, no. 10, 1456--1484 (2018).

\bibitem{RS4} N. Roidos, E. Schrohe. {\em The Cahn-Hilliard equation and the Allen-Cahn equation on manifolds with conical singularities}. Comm. Partial Differential Equations {\bf 38}, no. 5, 925--943 (2013).

\bibitem{RSS} N. Roidos, E. Schrohe, J. Seiler. {\em Bounded $H_\infty$-calculus for boundary value problems on manifolds with conical singularities.} J. Differential Equations {\bf297}, 370--408 (2021)

\bibitem{RSh1} N. Roidos, Y. Shao. {\em The fractional porous medium equation on manifolds with conical singularities I.} J. Evol. Equ. {\bf 22}, no. 1 (2022).

\bibitem{RSh2} N. Roidos, Y. Shao. {\em The fractional porous medium equation on manifolds with conical singularities II.} Math. Nachr. {\bf296}, no. 4, 1616--1650 (2023).
				
\bibitem{S24} E. Schrohe. {\em Introduction to the Analysis on Manifolds with Conical Singularities}. In: M. Chatzakou et al. (eds.), Modern Problems in PDEs and Applications, Research Perspectives Ghent Analysis and PDE Center 4, Springer Verlag 2024.

\bibitem{ScSc95} E.\ Schrohe, B.-W.\ Schulze. {\em Boundary value problems in Boutet de Monvel's calculus for manifolds with conical singularities II}. In M. Demuth, E. Schrohe, B.-W. Schulze (eds.), Boundary Value Problems, Schr\"odinger Operators, Deformation Quantization, Math. Topics, Vol. 8: Advances in Part. Diff. Equ., Akademie-Verlag 1995. 

\bibitem{SS2} E. Schrohe, J. Seiler. {\em Bounded $H_\infty$-calculus for cone differential operators}. J. Evolution Equations {\bf 18}, 1395--1425 (2018). 

\bibitem{SS1} E. Schrohe, J. Seiler. {\em The resolvent of closed extensions of cone differential operators}. Canad. J. Math. {\bf 57}, no. 4, 771--811 (2005).

\bibitem{Schu}B.-W. Schulze. {\em Pseudo-Differential Operators on Manifolds with Conical Singularities}. Studies in Mathematics and Its Applications {\bf 24}, North-Holland Publishing Co. (1991)

\bibitem{W} L. Weis. {\em Operator-valued Fourier multiplier theorems and maximal $L_{p}$-regularity.} Math. Ann. \textbf{319}, no. 4, 735-758 (2001). 

\end{thebibliography}
\end{document}